\documentclass[a4paper,twoside,12pt]{article}
\usepackage{amssymb,amsmath,amsthm,amsfonts,latexsym}
\usepackage{graphicx}
\usepackage[pdftex,bookmarks,colorlinks=false]{hyperref}
\usepackage{caption,subcaption}
\usepackage[hmargin=1.25in,vmargin=1.25in]{geometry}
\usepackage[auth-sc-lg,affil-sl]{authblk}
\newtheorem{theorem}{Theorem}[section]

\newtheorem{corollary}[theorem] {Corollary}
\newtheorem{definition}[theorem]{Definition}

\newtheorem{observation}[theorem]{Observation}
\newtheorem{problem}[theorem]{Problem}
\newtheorem{proposition}[theorem]{Proposition}
\newtheorem{remark}[theorem]{Remark}

\title{\textbf{\sc A Study on Topological Integer Additive Set-Labeling of Graphs}}

\author{N. K. Sudev}
\affil{\small Department of Mathematics\\ Vidya Academy of Science \& Technology \\ Thalakkottukara, Thrissur - 680501, Kerala, India.\\ E-mail: sudevnk@gmail.com}

\author{K. A. Germina}
\affil{\small PG \& Research Department of Mathematics\\ Mary Matha Arts \& Science College\\Mananthavady, Wayanad-670645, Kerala, India.\\ E-mail: srgerminaka@gmail.com}

\date{}

\begin{document}
\maketitle

\begin{abstract}
A set-labeling of a graph $G$ is an injective function $f:V(G)\to \mathcal{P}(X)$ such that the induced function $f^{\oplus}:E(G)\to \mathcal{P}(X)-\{\emptyset\}$ defined by $f^{\oplus}(uv) = f(u){\oplus}f(v)$ for every $uv{\in} E(G)$, where $X$ is a non-empty finite set and $\mathcal{P}(X)$ be its power set. A set-indexer of $G$ is a set-labeling such that the induced function $f^{\oplus}$ is also injective. A set-indexer $f:V(G)\to \mathcal{P}(X)$ of a given graph $G$ is called a topological set-labeling of $G$ if $f(V(G))$ is a topology of $X$.  An integer additive set-labeling is an injective function $f:V(G)\to \mathcal{P}(\mathbb{N}_0)$, whose associated function $f^+:E(G)\to \mathcal{P}(\mathbb{N}_0)$ is defined by $f^+(uv)=f(u)+f(v), uv\in E(G)$, where $\mathbb{N}_0$ is the set of all non-negative integers. An integer additive set-indexer is an integer additive set-labeling such that the induced function $f^+:E(G) \to \mathcal{P}(\mathbb{N}_0)$ is also injective. In this paper, we extend the concepts of topological set-labeling of graphs to topological integer additive set-labeling of graphs.
\end{abstract}

\vspace{0.05in}

\section{Introduction}

For all  terms and definitions of graphs and graph classes, not defined specifically in this paper, we refer to \cite{BM1}, \cite{BLS}, \cite{FH} and \cite{DBW} and for graph labeling concepts, we refer to \cite{JAG}. For terms and definitions in topology, we further refer to \cite{JD}, \cite{KDJ1} and \cite{JRM}. Unless mentioned otherwise, all graphs considered here are simple, finite and have no isolated vertices. 

Research on graph labeling commenced with the introduction of $\beta$-valuations of graphs in \cite{AR}. Analogous to the number valuations of graphs, the concepts of set-assignments,  set-labelings and set-indexers of graphs are introduced in \cite{BDA1} as follows.

Let $G(V,E)$ be a given graph. Let $X$ be a non-empty set and $\mathcal{P}(X)$ be its power sets. Then, the set-valued function $f:V(G)\to \mathcal{P}(X)$ is called the {\em  set-assignment} of vertices of $G$ respectively. In a similar way, we can define a set assignment of edges of $G$ as a function $g:E(G)\to \mathcal{P}(Y)$ and a set assignment of elements (both vertices and edges) of $G$ as a function $h:V(G)\cup E(G)\to \mathcal{P}(Z)$, where $Y$ and $Z$ are non-empty sets. The term set assignment is used for set-assignment of vertices unless mentioned otherwise. 

A set-assignment of a graph $G$ is said to be a {\em set-labeling} or a {\em set-valuation} of $G$ if it is injective. A graph with a set-labeling $f:V(G)\to \mathcal{P}(X)$ is denoted by $(G,f)$ and is referred to as a {\em set-labeled graph} or a {\em set-valued graph}. 

For a graph $G(V,E)$ and a non-empty set $X$ of cardinality $n$, a {\em set-indexer} of $G$ is defined as an injective set-valued function $f:V(G) \to \mathcal{P}(X)$ such that the function $f^{\oplus}:E(G)\to \mathcal{P}(X)-\{\emptyset\}$ defined by $f^{\oplus}(uv) = f(u ){\oplus}f(v)$ for every $uv{\in} E(G)$ is also injective, where $\mathcal{P}(X)$ is the set of all subsets of $X$ and $\oplus$ is the symmetric difference of sets. A graph that admits a set-indexer is called a {\em set-indexed graph}. It is proved in \cite{BDA1} that every non-empty graph has a set-indexer. 

More studies on set-labeled and set-indexed graphs have been done in \cite{GK1}, \cite{BDA1}, \cite{AGAS} and \cite{BDA2}. Then, the notion of topological set-labeling of a graph is defined in \cite{AGPR} as follows.

Let $G$ be a graph and let $X$ be a non-empty set. A set-labeling $f:V(G)\to \mathcal{P}(X)$  is called a {\em topological set-labeling} of $G$ if $f(V(G))$ is a topology of $X$. A graph $G$ which admits a topological set-labeling is called a {\em topologically set-labeled graph}. More studies on topological set-labeling of different graphs have been done subsequently.

The {\em sumset} of two non-empty sets $A$ and $B$, denoted by $A+B$, is the set defined by $A+B=\{a+b: a\in A, b\in B\}$. For every non-empty set $A$, we have $A+\{0\}=A$. Hence, $\{0\}$ and $A$ are said to be the {\em trivial summands} of the set $A$. If $C=A+B$, where $A$ and $B$ are non-trivial summands of $C$, then $C$ is said to be the {\em non-trivial sumset} of $A$ and $B$. In this paper, by the terms sumsets and summands, we mean non-trivial sumsets and non-trivial summands respectively. 

If any either $A$ or $B$ is countably infinite, then their sumset $A+B$ will also be a countably infinite set. Hence, all sets mentioned in this paper are finite sets.We denote the cardinality of a set $A$ by $|A|$ and the power set of a set $A$ by $\mathcal{P}(A)$. We also denote, by $X$, the finite ground set of non-negative integers that is used for set-labeling the elements of $G$.

Using the terminology and concepts of sumset theory, a particular type of set-labeling, called integer additive set-labeling, was introduced as follows.

Let $\mathbb{N}_0$ be the set of all non-negative integers. An {\em integer additive set-labeling} (IASL, in short) is an injective function $f:V(G)\to \mathcal{P}(\mathbb{N}_0)$ such that the associated function $f^+:E(G)\to \mathcal{P}(X)$ is defined by $f^+ (uv) = f(u)+ f(v)$ for any two adjacent vertices $u$ and $v$ of $G$. A graph $G$ which admits an IASL is called an IASL graph. 

An {\em integer additive set-labeling} $f$ is an integer additive set-indexer (IASI, in short) if the induced function $f^+:E(G) \to \mathcal{P}(\mathbb{N}_0)$ defined by $f^+ (uv) = f(u)+ f(v)$ is injective.  A graph $G$ which admits an IASI is called an IASI graph (see \cite{GA},\cite{GS0}).

Cardinality of the set-label of an element (a vertex or an edge) of a graph $G$ is called the {\em set-indexing number} of that element. An IASL (or an IASI) is said to be a $k$-uniform IASL (or $k$-uniform IASI) if $|f^+(e)|=k ~ \forall ~ e\in E(G)$. The vertex set $V(G)$ is called {\em $l$-uniformly set-indexed}, if all the vertices of $G$ have the set-indexing number $l$.

Motivated by the studies on topological set-labeling of graphs, we introduce the notion of topological integer additive set-labeling of graphs and study the structural properties and characteristics of the graphs which admit this type of set-labeling.

\section{Topological IASL-Graphs}

Note that no vertex of a given graph $G$ has the empty set $\emptyset$ as its set-labeling with respect to a given integer additive set-labeling. Hence, in this paper, we consider only non-empty subsets of the ground set $X$ for set-labeling the elements of $G$.

Analogous to topological set-labeling of graphs, we introduce the notion of topological integer additive set-labeling of certain graphs as follows.

\begin{definition}{\rm
Let $G$ be a graph and let $X$ be a non-empty set of non-negative integers. An integer additive set-labeling $f:V(G)\to \mathcal{P}(X)-\{\emptyset\}$  is called a {\em topological integer additive set-labeling} (TIASL, in short) of $G$ if  $f(V(G))\cup \{\emptyset\}$ is a topology of $X$. A graph $G$ which admits a topological integer additive set-labeling is called a {\em topological integer additive set-labeled graph} (in short, TIASL-graph).}
\end{definition}

The notion of a topological integer additive set-indexer of a given graph $G$ is introduced as follows.

\begin{definition}{\rm
A topological integer additive set-labeling $f$ is called a {\em topological integer additive set-indexer} (TIASI, in short) if the associated function $f^+:E(G)\to \mathcal{P}(X)$ defined by $f^+(uv)=f(u)+f(v); ~ u,v\in V(G)$, is also injective. A graph $G$ which admits an integer additive set-graceful indexer is called an {\em topological integer additive set-indexed graph} (TIASI-graph, in short).}
\end{definition}

\begin{remark}\label{R-IASGL1.1}{\rm
For a finite set $X$ of non-negative integers, let the given function  $f:V(G)\to \mathcal{P}(X)-\{\emptyset\}$ be an integer additive set-labeling on a graph $G$. Since the set-label of every edge $uv$ is the sumset of the sets $f(u)$ and $f(v)$, it can be observed that $\{0\}$ can not be the set-label of any edge of $G$. More over, since $f$ is a TIASL defined on $G$, $X$ must be the set-label of some vertex, say $u$, of $G$ and hence the set $\{0\}$ will be the set-label of a vertex, say $v$, and the vertices $u$ and $v$ are adjacent in $G$.}
\end{remark}

Let $f$ be a topological integer additive set-indexer of a given graph $G$ with respect to a non-empty finite ground set $X$. Then, $\mathcal{T}=f(V(G))\cup \{\emptyset\}$ is a topology on $X$. Then, the graph $G$ is said to be a {\em $f$-graphical realisation} (or simply \textit{$f$-realisation}) of $\mathcal{T}$. The elements of the sets $f(V)$ are called \textit{$f$-open sets} in $G$.

An interesting question that arises in this context is about the existence of an $f$-graphical realisation for a topology $\mathcal{T}$ of a given non-empty set $X$. Existence of graphical realisations for certain topologies of a given set $X$ is established in the following theorem.

\begin{theorem}
Let $X$ be a non-empty finite set of non-negative integers. A topology $\mathcal{T}$ of $X$, consisting of the set $\{0\}$ is graphically realisable.
\end{theorem}
\begin{proof}
Let $X$ be a non-empty finite set of non-negative integers and let $0\in X$. Consider a topology $\mathcal{T}$ of $X$ consisting of the set $\{0\}$. We need to construct a graph $G$ such that the vertices of $G$ have the (non-empty) set-labels taken from $\mathcal{T}$ in an injective manner. Let us proceed in this direction as explained below. Take a star graph $K_{1,|\mathcal{T}|-2}$. Label its central vertex by $\{0\}$ and label the other vertices by the remaining $|\mathcal{T}|-2$ non-empty open sets in $\mathcal{T}$. Clearly, this labeling is a TIASL defined on the graph $K_{1,|\mathcal{T}|-2}$ and hence $K_{1,|\mathcal{T}|-2}$ is a graphical realisation of $\mathcal{T}$. 
\end{proof}

It can also be observed that if we join two vertices $u$ and $v$ of the above mentioned TIASL-graph $K_{1,|\mathcal{T}|-2}$ by an edge, subject to the condition that $f(u)+f(v)\subseteq X$, the resultant graph will also be a graphical realisation of $\mathcal{T}$. Hence, there may exist more than one graphical realisations for a given topology of $X$. In view of this fact, we have to address the questions regarding the structural properties of TIASL-graphs. Hence, we proceed to find out the structural properties of TIASL-graphs.

\begin{proposition}\label{P-TIASI1}
If $f:V(G)\to \mathcal{P}(X)-\{\emptyset\}$ is a TIASL of a graph $G$, then $G$ must have at least one pendant vertex.
\end{proposition}
\begin{proof}
Let $f$ be a TIASL defined on a graph $G$. Then, clearly  $X\in f(V)$. That is, for some vertex $v\in V(G),~ f(v)=X$. Then, by Remark \ref{R-IASGL1.1}, $v$ is adjacent to a vertex whose set-label is $\{0\}$. Now we claim that, the vertex $v$ can be adjacent to only one vertex that has the set-label $\{0\}$. This can be proved as follows. 

Let $u$ be a vertex that is adjacent to the vertex $v$ and let $a$ be a non-zero element of $X$. Also, let $l$ be the maximal element of $X$. If possible, let $a\in f(u)$. Then, the element $a+l \in f^+(uv)$ and is greater than $l$, which leads to a contradiction to the fact that $f^+(uv)=f(u)+f(v)\subseteq X$, as $f$ is an IASL of $G$. Therefore, the vertex of $G$ having the set-label $X$ can be adjacent to a unique vertex that has the set-label $\{0\}$. That is, the vertex $v$ with $f(v)=X$ is definitely a pendant vertex of $G$. Any TIASL-graph $G$ has at least one pendant vertex. 
\end{proof}

\noindent Figure \ref{G-TIASL} depicts the TIASL, say $f$, of a graph $G$, with respect to a ground set $X=\{0,1,2,3,4\}$ and a topology  $\mathcal{T}=\{\emptyset,X, \{0\},\{1\},\{2\},\{0,1\},\{0,2\},\{1,2\}, \{0,1,2\}\}$ of $X$, where $f(V(G))=\mathcal{T}-\{\emptyset\}$ is the collection of the set-labels of the vertices in $G$.

\begin{figure}[h!]
\centering
\includegraphics[scale=0.65]{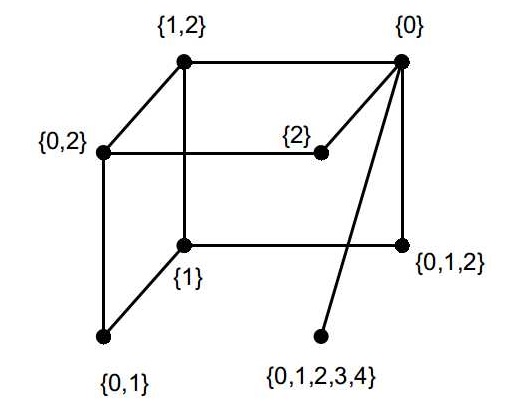}
\caption{An example to a TIASL-graph.}\label{G-TIASL}
\end{figure}

An interesting question in this context is about the number of pendant vertices required in a TIASL-graph. Clearly, the answer to this question depends on the ground set $X$ and the topology $\mathcal{T}$ of $X$ we choose for labeling the vertices of $G$. Our next objective is to determine the minimum number of pendant vertices required in a TIASL-graph.

\begin{proposition}\label{P-TIASI2}
Let $f:V(G)\to \mathcal{P}(X)-\{\emptyset\}$ is a TIASL of a graph $G$. Then, the vertices whose set-labels containing the maximal element of the ground set $X$ are pendant vertices which are adjacent to the vertex having the set-label $\{0\}$.
\end{proposition}
\begin{proof}
For  given ground set $X$ of non-negative integers, let $f:V(G)\to \mathcal{P}(X)-\emptyset$ be a TIASL of $G$. Let $l$ be the maximal element of the ground set $X$. Let $v$ be a vertex of $G$ whose set-label contains the element $l\in X$. Let $u$ be an adjacent vertex of $v$ whose set-label contains a non-zero element $b\in X$. Then, $b+l\not \in X$, contradicting the fact that $f$ is an IASL of $G$. If $l\in f(v) ~\text{for} ~ v\in V(G)$, then its adjacent vertices can have a set-label $\{0\}$. That is, all the vertices whose set-labels contain the maximal element of the ground set $X$ must be adjacent to a unique vertex whose set-label is $\{0\}$.
\end{proof}

\noindent Invoking Proposition \ref{P-TIASI1} and Proposition \ref{P-TIASI2}, we have

\begin{proposition}\label{P-TIASI3}
Let $X$ be the ground set and $\mathcal{T}$ be the topology of $X$ which are used for set-labeling the vertices of a TIASL-graph $G$. Then, an element $x_r$ in $X$ can be an element of the set-label $f(v)$ of a vertex $v$ of $G$ if and only if $x_r+x_s\le l$, where $x_s$ is any element of the set-label of another vertex $u$ which is adjacent to $v$ in $G$ and $l$ is the maximal element in $X$.
\end{proposition}

\noindent The following result is an immediate consequence of the above propositions.

\begin{proposition}\label{P-TIASI2a}
If $G$ has only one pendant vertex and if $G$ admits a TIASL, then $X$ is the only set-label of the vertices of $G$ containing the maximal element of $X$.
\end{proposition}

What is the minimum number of pendant vertices required for a graph which admits a TIASL with respect to a given topology $\mathcal{T}$ of the ground set $X$? The following proposition provides a solution to this question. 

\begin{proposition} 
Let $\mathcal{T}$ be a given topology of the ground set $X$. Then,
\begin{enumerate}\itemsep0mm
\item[(i)] the minimum number of pendant edges incident on a particular vertex of a TIASL-graph is equal to the number of $f$-open sets in $f(V(G))$ containing the maximal element of the ground set $X$
\item[(ii)] the minimum number of pendant vertices of a TIASL-graph $G$ is the number of $f$-open sets in $\mathcal{T}$, each of which is the non-trivial summand of at most one $f$-open set in $\mathcal{T}$.
\end{enumerate}
\end{proposition}
\begin{proof}
Let $G$ be a graph which admits a TIASL $f$ with respect to a topology $\mathcal{T}$ of the ground set $X$. 

\noindent{\em Case (i):} If an $f$-open set $X_i$ contains the maximal element of $X$, then by Proposition \ref{P-TIASI2}, $X_i$ can be the set-label of a pendant vertex, say $v_i$, which is adjacent to the vertex having set-label $\{0\}$. Hence, every $f$-open set containing the maximal element of $X$ must be the set-label of a pendant vertex that is adjacent to a single vertex whose set-label is $\{0\}$. Therefore, the minimum number of pendant edges incident on a single vertex is the number of $f$-open sets in $\mathcal{T}$ containing the maximal element of $X$.

\noindent {\em Case (ii):} If an $f$-open set $X_i$ is not a non-trivial summand of any $f$-open sets in $\mathcal{T}$, then the vertex with set-label $X_i$ can be adjacent only to the vertex with set-label $\{0\}$. If $X_i$ is the non-trivial summand of exactly one $f$-open set in $\mathcal{T}$, then the vertex $v_i$ with the set-label $X_i$ can be adjacent only to one vertex say $v_j$ with set-label $X_j$, where $X_i+X_j\subseteq X$. If $X_i$ is the non-trivial summand of more than one $f$-open sets in $\mathcal{T}$, then the vertex with set-label $X_i$ can be adjacent to more than one vertex of $G$ and hence $v_i$ need not be a pendant vertex. Therefore, the minimum number of pendant vertices in $G$ is the number of $f$-open sets in $\mathcal{T}$, each of which is the non-trivial summand of at most one $f$-open set in $\mathcal{T}$.
\end{proof}

Does every graph with one pendant vertex admit a TIASL? The answer to this question depends upon the choice of the ground set $X$. Hence, let us verify the existence of TIASL for certain standard graphs having pendant vertices by choosing a ground set $X$ suitably. For this, first consider the following graphs.

Let $G$ be a graph on $n$ vertices and let $P_m$ be a path that has no common vertex with $G$. We call the graph obtained by identifying one vertex of $G$ and one end vertex of $P_m$ an {\em $(n,m)$-ladle}. 

If $G$ is a cycle $C_n$, then this ladle graph is called an {\em (n,m)-tadpole graph} or a {\em dragon graph}. If $m=1$ in a tadpole graph, then $G$ is called an {\em $n$-pan}. 

If $G$ is a complete graph on $n$ vertices, then the corresponding $(n,m)$-ladle graph is called an {\em $(n,m)$-shovel}.

Now, we proceed to discuss the admissibility of TIASL by these types of graphs. The following result establishes the admissibility of TIASL by a pan graph.

\begin{proposition}\label{P-TIASL-p}
A pan graph admits a topological integer additive set-labeling.
\end{proposition}
\begin{proof}
Let $G$ be an $m$-pan graph. Let $v$ be the pendant vertex and $v_1,v_2, \ldots, v_n$ be the vertices of $C_n$. Without loss of generality, let $v_1$ be the unique vertex adjacent to $v$ in $G$. Label the vertices of the cycle $C_n$ of $G$ in such a way that we have $f(v_1)=\{0\}, f(v_i)=\{0,1,\ldots, i-1\}: 2\le i\le n$. Now, let $X=\{0,1,2,3,\ldots, m\}$, where $m\ge 2n-3$ and label the pendant vertex $v$ by the set $X$. Hence, the collection of the set-labels of the vertices of $G$ is $\mathcal{A}=\{\{0\}, \{0,1\}, \{0,1,2\}, \ldots, \{0,1,2,\ldots,n-1\}, X\}$. Clearly, the set $\mathcal{T}=\mathcal{A}\cup \{\emptyset\}$ is a topology on $X$. Therefore, this labeling of $G$ is a TIASL of $G$. Hence, the $n$-pan $G$ admits a TIASL. 
\end{proof}

\noindent Figure \ref{fig:G-TIASL4a} illustrates the admissibility of TIASL by an $n$-pan with respect to the ground set $X=\{0,1,2,3,\ldots, 2n-3\}$.

\begin{figure}[h!]
\centering
\includegraphics[width=0.75\linewidth]{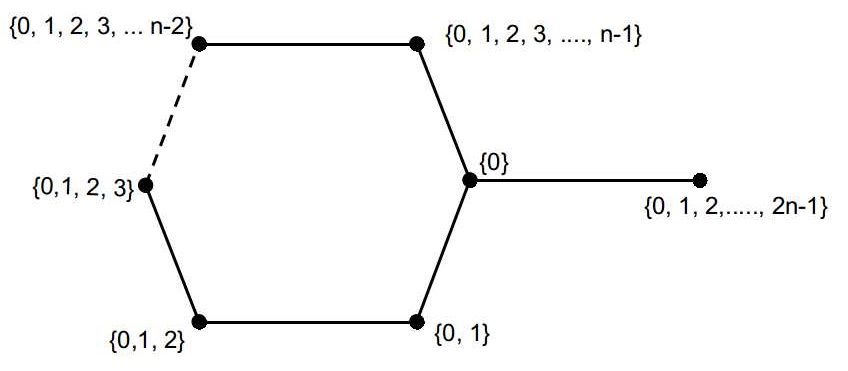}
\caption{An $n$-pan graph with a TIASL defined on it.}
\label{fig:G-TIASL4a}
\end{figure}

We now proceed to verify the admissibility of TIASL by the general tadpole graphs.

\begin{proposition}\label{P-TIASL-t}
A tadpole graph admits a topological integer additive set-labeling.
\end{proposition}
\begin{proof}
Let $G$ be an $(n,m)$-tadpole graph. Let $\{v_1,v_2, v_3, \ldots, v_n\}$ be the vertex set of $C_n$ and let $\{u_0,u_1,u_2,u_3, \ldots, u_m\}$ be the vertex set of $P_m$. Without loss of generality, let $u_0$ be the pendant vertex of $P_m$ in $G$. Identify the vertex $u_m$ of $P_m$ and the vertex $v_1$ of the cycle to form a tadpole graph. Let us define an IASL $f$ on $G$ as follows. Label the vertex $u_1$ by the set $\{0\}$, the vertex $u_2$ by the set $\{0,1\}$ and in general, the vertex $u_i$ by the set $\{0,1,2,\ldots, i-1\}$, for $1\le i\le m$. Therefore, the set-label of the vertex $u_m=v_1$ is $\{0,1,2,\ldots, m-1\}$. Now, label the remaining vertices of $C_n$ in $G$ as follows. Label the vertex $v_2$ by the set $\{0,1,2,\ldots, m\}$ and in general, label the vertex $v_j$ by the set $\{0,1,2,\ldots, m+j-2\}$.  Now, choose the set $X=\{0,1,2,\ldots,l\}$, where $l\ge 2(m+n)-5$. Now, the only vertex of $G$ that remains to be labeled is the pendant vertex. Label the  vertex $u_0$ by the set $X$. Then, the collection of set-labels of $G$ is $\mathcal{A}=\{\{0\}, \{0,1\}, \{0,1,2\}, \ldots, \{0,1,2,\ldots,m+n-2\},X\}$. Clearly, the set $\mathcal{T}=\mathcal{A}\cup \{\emptyset\}$ is a topology on $X$. Hence, this labeling is a TIASL defined on $G$. 
\end{proof}

\noindent Figure \ref{fig:G-TIASL4b} illustrates the admissibility of TIASL by the $(m,n)$-tadpole graph with respect to the ground set $X=\{0,1,2,\ldots,2(m+n)-5\}$.

\begin{figure}[h!]
\centering
\includegraphics[width=0.9\linewidth]{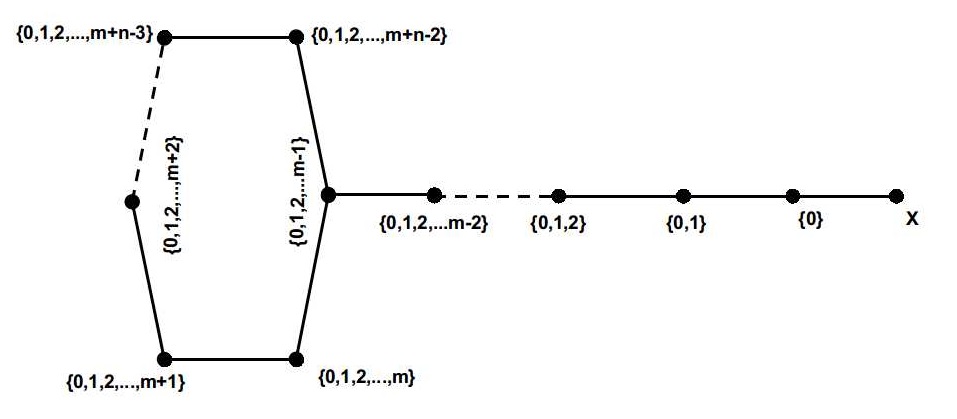}
\caption{An $(n,m)$-tadpole graph with a TIASL defined on it.}
\label{fig:G-TIASL4b}
\end{figure}

We can extend the above results to the shovel graphs also. The following result establishes the admissibility of TIASL by shovel graphs by properly choosing the ground set $X$.

\begin{proposition}\label{P-TIASL-s}
The $(n,m)$-shovel graph admits a topological integer additive set-labeling. 
\end{proposition}
\begin{proof}
Let $G$ be an $(n,m)$-shovel graph. Let $\{v_0,v_1,v_2,v_3, \ldots, v_m\}$ be the vertex set of $P_m$ and $\{v_m,v_{m+1}, v_{m+2}, \ldots, v_{m+n-1}\}$ be the vertex set of $K_n$ in the given shovel graph $G$, where $v_0$ is the pendant vertex of $P_m$ (and hence of $G$). Define an IASL $f$ on $G$ which assigns set-labels to the vertices of $G$ injectively in such a way that any vertex $v_i$ has the set-label $\{0,1,2,\ldots, i-1\}$, for $1\le i \le m+n-1$. Note that, the pendant vertex $v_0$ remains unlabeled at the moment. It can be noted that the maximal element of the set-label $f^+(v_{m+n-2}v_{m+n-1})$ is $2(m+n)-5$. Hence, choose the set $X=\{0,1,2,3,\ldots, 2(m+n)-5\}$ and label the pendant vertex $v_0$ by the set $X$ itself. Therefore, $f(V(G))=\{\{0\}, \{0,1\}, \{0,1,2\}, \ldots, \{0,1,2,\ldots,m+n-2\},X\}$ and $f(V(G))\cup \{\emptyset\}$ is a topology on $X$. Hence, $f$ is a TIASL on $G$.
\end{proof}

\noindent Figure \ref{fig:G-TIASL5} depicts the admissibility of TIASL by an $(n,m)$-shovel graph with ground set $X=\{0,1,2,3,\ldots, 2(m+n)-5\}$.

\begin{figure}[h!]
\centering
\includegraphics[width=0.85\linewidth]{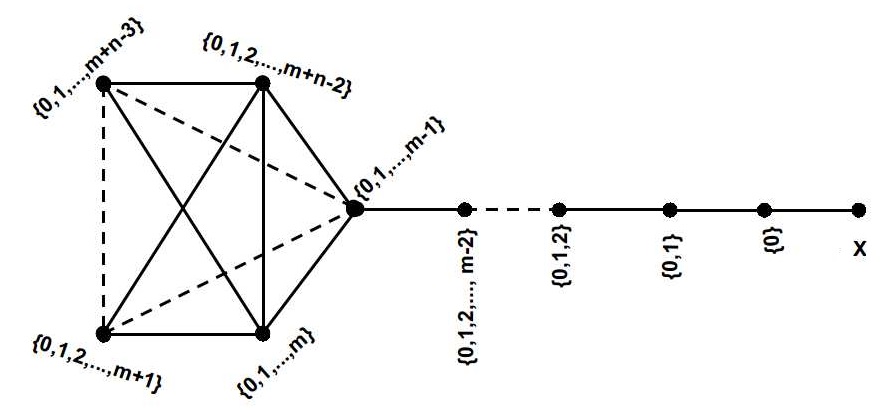}
\caption{An $(n,m)$-shovel graph with a TIASL defined on it.}
\label{fig:G-TIASL5}
\end{figure}

The above propositions raise the question whether the existence of a pendant vertex in a given graph $G$ results in the admissibility of TIASL by it. The choice of $X$ in all the above results played a major role in establishing a TIASL for $G$.  The following is a necessary and sufficient condition for a given graph with at least one pendant vertex to admit a TIASL.

\begin{theorem}\label{T-TIASI1}
A graph $G$ admits a TIASL if and only if $G$ has at least one pendant vertices. 
\end{theorem}
\begin{proof}
Let $G$ be a graph which admits a TIASL, say $f$. Then, the ground set $X\in f(V(G))$. Hence, by Proposition \ref{P-TIASI2}, the vertex with the set-label $X$ must be a pendant vertex. More over, by Proposition \ref{P-TIASI2}, if the set-label of a vertex $v_i$ contains the maximal element of $X$, then $v_i$ is a pendant vertex. Then, $G$ has at least one pendant vertex.

Conversely, assume that $G$ has at least one pendant vertex. Let $V(G)=\{v_1,v_2,v_3\ldots,v_n\}$. Without loss of generality, let $v_1$ be a pendant vertex of $G$. Now, label the vertex $v_i$ by the set $\{0,1,2,3, \ldots, i-1\}$ for $1\le i\le n$. Then, as explained in the above results, the maximal element in all set-labels of edges of $G$ is $2n-3$. Choose $X=\{0,1,2,\ldots, 2n-3\}$ and label the pendant vertex $v_1$ by the set $X$. Then, $f(V(G))=\{\{0\},\{0,1\},\{0,1,2\},\ldots, \{0,1,2,\ldots,n-1\}, X\}$. Therefore, $f(V(G))\cup \{\emptyset\}$ is a topology on $X$ and hence this labeling is a TIASL of $G$. 
\end{proof}

\noindent Theorem \ref{T-TIASI1} gives rise to the following result.

\begin{theorem}\label{T-TIASI1a}
Let $G$ be a graph with a pendant vertex $v$ which admits a TIASL, say $f$, with respect to a ground set $X$. Let $f_1$ be the restriction of $f$ to the graph $G-v$. Then, there exists a collection $\mathcal{B}$ of proper subsets of $X$ which together with $\{\emptyset\}$ form a topology of the union of all elements of $\mathcal{B}$.
\end{theorem}
\begin{proof}
Let $G$ be a graph with one pendant vertex, say $v$ and $X$ be the ground set for labeling the vertices of $G$. Choose the collection $\mathcal{B}$ of proper subsets of $X$ which contains the set $\{0\}$ and has the cardinality $n-1$ such that the sum of the maximal elements of any two sets in it is less than or equal to the maximal element of $X$ and the union of any two sets and the intersection of any two non-singleton sets in $\mathcal{B}$ are also in $\mathcal{B}$. Then, by Theorem \ref{T-TIASI1}, the set-labeling $f$ under which the pendant vertex $v$ is labeled by the set $X$ and other vertices of $G$ are labeled by the elements of $\mathcal{B}$ is a TIASL of $G$.

Let $f_1$ be the restriction of $f$ to the graph $G-v$. Therefore, $\mathcal{B}=f_1(V(G-v))$. Now let $B=\bigcup_{B_i\in \mathcal{B}}B_i$ and let $\mathcal{T'}=\mathcal{B}\cup \{\emptyset\}$. Since $G$ has only one end vertex, by Proposition \ref{P-TIASI2a}, no element of $\mathcal{A}$ contains the maximal element of $X$. Therefore, $B$ also does not contain the maximal element of $X$. Since the union of any number of sets in $\mathcal{B}$ is also in $\mathcal{B}$, the union of the elements in $\mathcal{T'}$. Then, $B$ belongs to $\mathcal{B}$ and to $\mathcal{T'}$ and $B$ is the maximal element of $\mathcal{T'}$. Since the intersection of any two non-singleton sets in $\mathcal{B}$ is also in $\mathcal{B}$ and $\emptyset \in \mathcal{T'}$, the finite intersection of elements in $\mathcal{T'}$ is also in $\mathcal{T'}$. The set $\mathcal{T'}=\mathcal{B}\cup \{\emptyset\}$ is a topology of the maximal set $B$ in $\mathcal{B}$.
\end{proof}

\begin{remark}{\rm 
If $v$ is the only pendant vertex of a given graph $G$, then the collection $\mathcal{B}=f(V(G-v))$, chosen as explained in Theorem \ref{T-TIASI1a} does not induce a topological IASL on the graph $G-v$, since $f^+(uw)\neq f(u)+f(w)$, for some edge $uw\in E(G-v)$.}
\end{remark}

\section{TIASLs with respect to Certain Topologies}

The number of elements in the ground set $X$ is very important in all the studies of set-labeling of graphs. Keeping this in mind, we define

\begin{definition}{\rm 
The minimum cardinality of the ground set $X$ required for a given graph to admit a topological IASL is known as the {\em topological set-indexing number} (topological set-indexing number) of that graph.}
\end{definition}

In this section, we discuss the existence and admissibility of topological IASLs with respect to some standard topologies like indiscrete topologies and discrete topologies.

A topology $\mathcal{T}$ is said to be an indiscrete topology of $X$ if $\mathcal{T}=\{\emptyset,X\}$. Hence the following result is immediate.

\begin{theorem}\label{T-TIASI4}
A graph $G$ admits a TIASL with respect to the indiscrete topology $\mathcal{T}$ if and only if $G\cong K_1$.
\end{theorem}
\begin{proof}
Let $v$ be the single vertex of the graph  $G=K_1$. Let $X$ be the ground set for set-labeling $G$. Let $f(v)=X$. Then $f(V)=\{X\}$ and $f(V)\cup \{\emptyset\}=\{\emptyset, X\}$, which is the indiscrete topology on $X$. Conversely, assume that $G$ admits a TIASL with respect to the indiscrete topology $\mathcal{T}$ of the ground set $X$. Then, $f(V(G))=\mathcal{T}-\{\emptyset\}=\{X\}$, a singleton set. Therefore, $G$ can have only a single vertex. That is, $G\cong K_1$.
\end{proof}

\noindent From Proposition \ref{T-TIASI4}, we have the following result.

\begin{proposition}\label{P-TIASI4a}
The topological set-indexing number of $K_1$ is $1$.
\end{proposition}

Another basic topology of a set $X$ is the {\em Sierpenski's topology}. If $X$ is a two point set, say $X=\{0,1\}$, then the topology $\mathcal{T}_1=\{\emptyset, \{0\},X\}$ and  $\mathcal{T}_2=\{\emptyset, \{0\},X\}$ are the Sierpenski's topologies. The following result establishes the conditions required for a graph to admit a TIASL with respect to the Sierpenski's topology. 

\begin{theorem}\label{T-TIASI5}
A graph $G$ admits a TIASL with respect to the Sierpenski's topology if and only if $G\cong K_2$.
\end{theorem}
\begin{proof}
Let $G$ be the given graph, with vertex set $V$, which admits a TIASL with respect to the Sierpenski's topology. Let a two point set $X=\{0,1\}$ be the ground set used for set-labeling the graph $G$. Then, $f(V)=\{\{0\},X\}$. Therefore, $G$ can have exactly two vertices. That is, $G\cong K_2$.

Conversely, assume that $G\cong K_2$. Let $u$ and $v$ be the two vertices of $G$. Choose a two point set $X$ as the ground set to label the vertices of $G$. Label the vertex $u$ by $X$. Then by Proposition \ref{P-TIASI1}, $v$ must have the set-label $\{0\}$. Then $f(V(G))=\{\{0\},X\}$. Then, $f(v(G))\cup \{\emptyset\}$ is a topology on $X$, which is a Sierpenski's topology of $X$. Therefore, $G\cong K_2$ admits a TIASL with respect to the Sierpenski's topology.
\end{proof}

From the above result, we observe the following.
\begin{observation}
The only Sierpenski's topology of the two point set $X=\{0,1\}$ that induces a TIASL on the graph $K_2$ is $\mathcal{T}=\{\emptyset,\{0\},X\}$.
\end{observation}

In view of Proposition \ref{T-TIASI5}, we claim that for any ground set $X$ containing two or more elements, one of which is $0$, induces a TIASL on $K_2$. Therefore,  the following result is immediate.

\begin{proposition}
The topological set-indexing number of $K_2$ is $2$.
\end{proposition}

\noindent The following results are the immediate consequences of \ref{T-TIASI1}.

\begin{proposition}\label{P-TIASI6}
For $n\ge3$, no complete graph $K_n$ admits a TIASL.
\end{proposition}
\begin{proof}
The proof follows from Theorem \ref{T-TIASI1} and from the fact that a complete graph on more than two vertices does not have any pendant vertex.
\end{proof}

\begin{proposition}\label{P-TIASI7}
For $m,n\ge 2$, no complete bipartite graph $K_{m,n}$ admits a TIASL.
\end{proposition}
\begin{proof}
The proof is immediate from the fact that a complete bipartite graph has no pendant vertices.
\end{proof}

\begin{corollary}
A path graph $P_m$ admits a TIASL.
\end{corollary}
\begin{proof}
Every path graph $P_m$ has two pendant vertices and hence satisfy the condition mentioned in Theorem \ref{T-TIASI1}. Hence $P_m$ admits a TIASL.
\end{proof}

\begin{proposition}
Every tree admits a TIASL.
\end{proposition}
\begin{proof}
Since every tree $G$ has at least two pendant vertices, by Theorem \ref{P-TIASI1}, $G$ admits a TIASL.
\end{proof}

\begin{proposition}\label{P-TIASI8}
No cycle graph $C_n$ admits a TIASL.
\end{proposition}
\begin{proof}
A cycle does not have any pendant vertex. Then, the proof follows immediately by Theorem \ref{T-TIASI1}.
\end{proof}

In view of the above results, we arrive at the following inference.

\begin{proposition} 
For $k\ge 2$, no $k$-connected graph admits a TIASL with respect to a ground set $X$. 
\end{proposition}
\begin{proof}
No biconnected graph $G$ can have pendant vertices. Hence, by Theorem \ref{T-TIASI2}, $G$ can not admit a TIASL.
\end{proof}

We have already discussed the admissibility of a TIASL by a graph with respect to the indiscrete topology of the ground set $X$.  In this context, it is natural to ask whether a given graph admits the TIASL with respect to the discrete topology of a given set $X$. The following theorem establishes the condition required for $G$ to admit a TIASL with respect to the discrete topology of $X$. 

\begin{theorem}\label{T-TIASI2}
A graph $G$, on $n$ vertices, admits a TIASL with respect to the discrete topology of the ground set $X$ if and only if $G$ has at least $2^{|X|-1}$ pendant vertices which are adjacent to a single vertex of $G$.
\end{theorem}
\begin{proof}
Let $|X|=m$. Let the graph $G$ admits a TIASL $f$ with respect to the discrete topology $\mathcal{T}$ of $X$. Therefore, $f(V(G))=\mathcal{P}(X)-\{\emptyset\}$. Then, $|f(V(G))|=2^{|X|}-1$. Now, let $l$ be the maximal element in $X$. The number of subsets of $X$ containing $l$ is $2^{m-1}$. Since  $f$ is a TIASL with respect to the discrete topology, all these sets containing $l$ must also be the set-labels of some vertices of $G$. By Proposition \ref{P-TIASI2}, all these vertices must be adjacent to the vertex whose set-label is $\{0\}$. By Proposition \ref{P-TIASI3}, no two vertices whose set-labels contain $l$ can be adjacent among themselves or to any other vertex which has a set-label with non-zero elements. Therefore, $G$ has $2^{m-1}$ pendant vertices which are adjacent to a single vertex whose set-label is $\{0\}$.

Conversely, let $G$ be a graph with $n=2^{|X|}-1$ vertices such that at least $2^{|X|-1}$ of them are pendant vertices incident on a single vertex of $G$. Label these pendant vertices by the $2^{|X|-1}$ subsets of $X$ containing the maximal element $l$ of $X$. Label remaining vertices of $G$ by the remaining $2^{|X|-1}-1$ subsets of $X$ which do not contain the element $l$, in such a way that the sum of the maximal elements of the set-labels of two adjacent vertices is less than or equal to $l$. This labeling is clearly a TIASL on $G$. That is, $G$ admits a TIASL with respect to the discrete topology of $X$.  
\end{proof}

Figure \ref{G-TIASL2a} depicts the existence of a TIASL with respect to the discrete topology of a ground set $X=\{0,1,2,3\}$ for a graph $G$.

\begin{figure}[h!]
\centering
\includegraphics[scale=0.45]{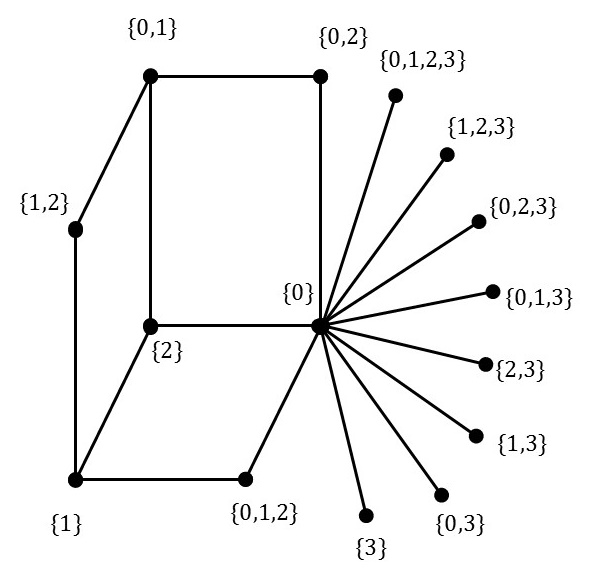}
\caption{\small \sl A TIASL of graph with respect to the discrete topology of $X$.}\label{G-TIASL2a}
\end{figure}

Since the necessary and sufficient condition for a graph to admit a TIASL with respect to the discrete topology of  ground set $X$ is that $G$ has at least $2^{|X|-1}$ pendant vertices that incident at a single vertex of $G$, no paths $P_n; ~n\ge 3$, cycles, complete graphs and complete bipartite graphs can have a TIASL with respect to discrete topology of $X$. 

\noindent Theorem \ref{P-TIASI2} gives rise to the following results also. 

\begin{corollary}
A graph on even number of vertices does not admit a TIASL with respect to the discrete topology of the ground set $X$.
\end{corollary}
\begin{proof}
If a graph on $n$ vertices admits a TIASL with respect to the discrete topology of the ground set $X$, then by Theorem \ref{T-TIASI2}, $n=2^{|X|-1}$, which can never be an even integer. Therefore, $G$ on even number of vertices does not admit a TIASL with respect to the discrete topology of $X$.
\end{proof}

\begin{corollary}
A star graph $K_{1,r}$ admits a TIASL with respect to the discrete topology of the ground set $X$, if and only if $r=2^{|X|}-2$.
\end{corollary}
\begin{proof}
First assume that the star graph $G=K_{1,r}$ admits a TIASL $f$ with respect to the discrete topology of the ground set $X$. Then, $f(V(G))=\mathcal{P}(X)-\{\emptyset\}$. That is, $|f(V(G))|=2^{|X|}-1$. Hence, $G$ must have $2^{|X|}-1$ vertices. That is, $r+1=2^{|X|}-1$. Therefore, $r=2^{|X|}-2$.

Conversely, consider a star graph $G=K_{1,r}$, where $r=2^n-2$ for some positive integer $n$. Choose a set $X$ with cardinality $n$, which consists of the element $0$. Note that the number of non-empty subsets of $X$ is $2^n-1$. Define a set-labeling $f$ of $G$ which assigns $\{0\}$ to the central vertex of $G$ and the other non-empty subsets of $X$ to the pendant vertices of $G$. Clearly, this labeling is an IASL of $G$. Also, $f(V(G))=\mathcal{P}(X)-\{\emptyset\}$. Therefore, $f$ is a TIASL of $G$ with respect to the discrete topology of $X$.
\end{proof}

Figure \ref{G-TIASL2b} illustrates the existence of a TIASL with respect to the discrete topology of the ground set $X$ for a star graph.

\begin{figure}[h!]
\begin{center}
\includegraphics[scale=0.35]{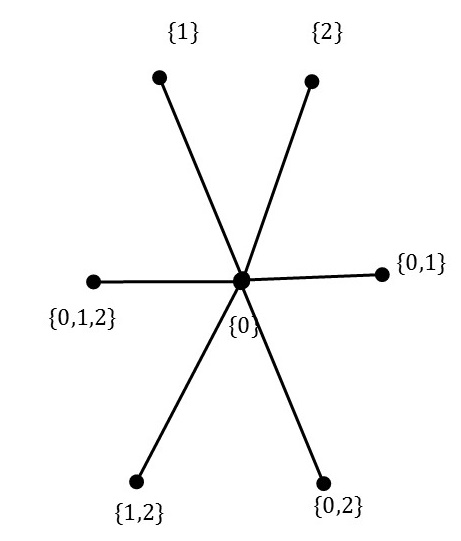}
\caption{\small \sl }\label{G-TIASL2b}
\caption{\small A Star grph with a TIASL with respect to the discrete topology of $X$.}
\end{center}  
\end{figure}

\section{Conclusion}

In this paper, we have discussed the concepts and properties of topological integer additive set-indexed graphs analogous to those of topological IASI graphs and have done a characterisation based on this labeling.

We note that the admissibility of topological integer additive set-indexers by the given graphs depends also upon the number and nature of the elements in $X$ and the topology $\mathcal{T}$ of $X$ concerned.  Hence, choosing a ground set $X$ is very important in the process of checking whether a given graph admits a TIASL-graph.

Certain problems in this area are still open. Some of the areas which seem to be promising for further studies are listed below.

\begin{problem}{\rm
Characterise different graph classes which admit topological integer additive set-labelings.}
\end{problem}

\begin{problem}{\rm
Estimate the topological set-indexing number of different graphs and graph classes which admit topological integer additive set-labelings.} 
\end{problem}

\begin{problem}{\rm
Verify the existence of topological integer additive set-labelings for different graph operations and graph products.} 
\end{problem}

\begin{problem}{\rm
Establish the necessary and sufficient condition for a graph to admit topological integer additive set-indexer.}
\end{problem}

\begin{problem}{\rm
Characterise the graphs and graph classes which admit TIASI.}
\end{problem}

The integer additive set-indexers under which the vertices of a given graph are labeled by different standard sequences of non negative integers, are also worth studying. All these facts highlight a wide scope for further studies in this area.

\section*{Acknowledgement}

The authors would like to thank the anonymous reviewer for his/her insightful suggestions and critical and constructive remarks which made the overall presentation of this paper better.

\end{document}